\UseRawInputEncoding

\documentclass[11pt]{article}
\usepackage{amssymb,amsmath,amsthm,amsfonts}
\usepackage[dvips]{graphicx}
\usepackage{graphicx}

\theoremstyle{plain}
\newtheorem{theorem}{Theorem}[section]
\newtheorem{lemma}{Lemma}[section]

\numberwithin{equation}{section}

\theoremstyle{definition}

\newtheorem{remark}{Remark}[section]

\newcommand{\R}{{\mathbb R}}

\begin{document}
\title{{\bf  Observation estimates for a semilinear heat equation in $\mathbb{R}^n$}\footnote{\small
This work was partially supported by
Zhejiang Provincial Natural Science Foundation of China under
Grant No. LY24A010014 and LZ21A010001.} }

\author{Guojie Zheng$^{a}$,  \; Xin Yu$^{b}$
\\
\\
$^a${\it College of Digital Technology and Engineering}\\
{\it Ningbo University of Finance $\&$  Economics, Ningbo, 315175, China}\\
$^a${\it  Institute of Information Processing and Intelligent Computing}\\
{\it Ningbo Tech University, Ningbo, 315199, China}}

\date{}

 \maketitle

\begin{abstract}
This paper studies the state observation problems for the  semilinear heat equation in $\mathbb{R}^n$.
We derive observation estimates for the equation using the logarithmic convexity property  of the frequency function (see \cite{Phung}).
As an application, we show that if two solutions  coincide on  a nonempty open subset  $\omega\subset\Omega$  at some time $T>0$, then
they must be identical.

\vspace{0.3cm}


\vspace{0.3cm}

\end{abstract}

\section{Introduction}\label{sec1}
\ \ \ \
Let $T>0$ be a positive number.
 This paper studies the following semilinear heat equation in $\mathbb{R}^n$:
 \begin{equation}\label{1.1}
   \left\{
\begin{array}{ll}
\partial_t y-\triangle y+f\big(y\big)=0,& \textrm{ in } \mathbb{R}^{n}\times(0,T],\\
y(x,0)=y_0(x),    &   \textrm{ in } \mathbb{R}^{n},
\end{array}
\right.
 \end{equation}
 where the unknown function $y=y(x,t)$ depends on the space variable $x$ and the time variable $t$,
$y_0(x)$ is the initial data and  $f(\cdot): \mathbb{R}\mapsto\mathbb{R}$ is  a given nonlinear function.
 This model plays a crucial role in PDEs theory due to its applications in physics, biology, and finance
 (see \cite{Friedman, Peng, Zou}).
 This equation describes heat conduction with nonlinear sources, population dynamics, and diffusion processes.
  The behavior of the solution is significantly affected by the nonlinear term $f(y)$.

This paper studies the state observation problems for  equation (\ref{1.1}).
The key question is whether we can determine the solution to this equation from partial measurements of the solution in the observation region.
Consider equation (\ref{1.1}) with observation region taken in a subdomain $\omega\subset \mathbb{R}^n$.
We aim to estimate $y$ over the entire space $\mathbb{R}^{n}$ at time $T$ using observation data.
The nonlinearity makes it harder to analyze than the linear case.
This question is motivated by
 the literature \cite{Phung1},  which studies a similar equation in a bounded convex domain.
 Extensive related references can be found in \cite{Buffe,  Wanglj, Lu, PW} and the rich works cited therein.
Reference \cite{Wang1} studies the relationship  between the observation region and observability for the linear heat equation
in $\mathbb{R}^n$ and shows that an observation region satisfies the observability inequality if and only if it is $\gamma-$thick at scale $L$ for some $\gamma>0$ and $L>0$. To study this problem, we need to characterize suitable conditions for the observation region $\omega$.

\textit{Notation:}
Throughout this article $B_{\rho}(x_0)$ (where $\rho>0$ and $x_0\in\mathbb{R}^n$) denotes the open ball centered at $x_{0}$ and with radius $\rho$.
Let $\mathbb{N}=\{0,1,2,\ldots\}$ and $\mathbb{N}^{+}=\mathbb{N}\setminus \{0\}$. Denote by $\overline{\Omega}$ the closure of the subset $\Omega\subset\mathbb{R}^n$.
Let $\omega$ be an open subset of $\mathbb{R}^n$ satisfying the following condition ($A_1$).
\begin{itemize}
\item[($A_1$)]
There exists a sequence of open cubes with the same size $\{I_j\}_{j\in\mathbb{N}^{+}}$ of $\mathbb{R}^n$ such that\\
$(i)$ $I_j\cap I_k=\emptyset$, when $j\neq k$;\\
$(ii)$ $\overline{\cup_{j\in\mathbb{N}^{+}}I_j}=\mathbb{R}^n$;\\
$(iii)$ There exists a positive number $r$ satisfying for each $j\in \mathbb{N}^{+}$ there is a $x_j\in I_j$ such that $B_{r}(x_j)\subset I_j\cap\omega$.
\end{itemize}
\begin{remark}
Let $L$ be the side length of the cubes $I_j$ $(j\in \mathbb{N}^+)$. It is easy to check that $\omega$ is $\gamma-$thick at scale $L$ for some $\gamma>0$ (see \cite{Egidi, Wang1} for more details).
\end{remark}

The nonlinear term $f(y)$ complicates the analysis of (\ref{1.1}), and special techniques are required to handle it.
To this end, we give an
assumption on the function $f(\cdot)$ as follows.

\begin{itemize}
\item[($A_2$)] The function $f(\cdot): \mathbb{R}\mapsto\mathbb{R}$ is locally Lipschitz and satisfies
 $f(0)=0$.
\end{itemize}

The first results are stated below.
\begin{theorem}\label{theorem1}
Suppose that $(A_1)$ and $(A_2)$ hold.
Let $y_i^0\in L^2(\mathbb{R}^n)\cap L^\infty(\mathbb{R}^n)$ $(i=1,2)$ be the initial data and let
 $y_i \in L^\infty(0,T;L^\infty(\mathbb{R}^n))$  $(i=1,2)$ be the solution  to the system (\ref{1.1}) with the initial value
$y_i^0$. Let $L_{M}>0$ be the Lipschitz constant of $f(\cdot)$ on the domain
$\mathcal{D}_{M}:=\{s\in \mathbb{R}\:|\: |s|\leq M \},$ where
 \begin{eqnarray}\label{M}
M:=\max\{\|y_i\|_{L^\infty(0, T; L^{\infty}(\mathbb{R}^n))}\:|\:i=1,2\}.
 \end{eqnarray}
 Then, we obtain the following estimates:\\
$(i)$ There exist two positive numbers  $\beta=\beta(L, r)\in(0,1)$ and  $C=C(L, r, T, L_M)$,
  such that
\begin{eqnarray}\label{1.3}
\int_{\mathbb{R}^n}|y_1-y_2|^2(x,T)dx
&\leq& \!\!C\big(\int_{\mathbb{R}^n}|y_1^0-y_2^0|^2(x)dx\big)^{1-\beta}\big(\int_{\omega}|y_1-y_2|^2(x,T)dx\big)^\beta.
\end{eqnarray}
$(ii)$ If $y_1^0\neq y_2^0$, then there exists a positive number  $C=C(L, r, T, L_M)$ such that
\begin{eqnarray}\label{1.4}
\int_{\mathbb{R}^n}|y_1^0-y_2^0|^2(x)dx &\leq& C\exp\big(C\frac{\|y_1^0-y_2^0\|_{L^2(\mathbb{R}^n)}^2}{\|y_1^0-y_2^0\|_{H^{-1}(\mathbb{R}^n)}^2}\big)\int_{\omega}|y_1-y_2|^2(x,T)dx.
\end{eqnarray}
\end{theorem}
\begin{remark}
{\it
Several notes on this theorem  are given in order.

\begin{itemize}
  \item[(a)] Estimate (\ref{1.3}) is called the global interpolation inequality for equation (\ref{1.1}), which is a powerful tool in the analysis of state observation problems.

  \item[(b)] Equation (\ref{1.4}) allows us to reconstruct the solution of (\ref{1.1}) from partial observations.
    We call it an observation estimate for equation (\ref{1.1}). It provides a framework to relate the state of the system to partial observations. These estimates are essential in control theory.

    \item[(c)] Observation estimate can be used to prove unique continuation, which means that if the solution is identically zero in $\omega$ at time $T$, then it must be identically zero everywhere.
\end{itemize}}
\end{remark}

Next, we introduce another assumption on $f(\cdot)$.
\begin{itemize}
\item[($A_3$)] The function $f(\cdot): \mathbb{R}\mapsto\mathbb{R}$ belongs to $C^1(\mathbb{R})$ and satisfies  $f(0)=0$. Moreover, there are positive numbers $C>0$ and $p>1$
such that for any $x_1, x_2\in\mathbb{R}$,
\begin{eqnarray*}
|f(x_1)-f(x_2)|\leq C(|x_1|^{p-1}+|x_2|^{p-1})|x_1-x_2|.
\end{eqnarray*}
\end{itemize}
With its help, we have
\begin{theorem}\label{theorem2}
Suppose that $(A_1)$ and $(A_3)$ hold. Let $p<1+\frac{4}{n}$, and
let $y_i\in C([0,T^*]; L^2(\mathbb{R}^n))\cap L_{loc}^\infty((0,T); L^\infty(\mathbb{R}^n))$ ($i=1,2$) be the solution  to the system (\ref{1.1}) corresponding to the initial value $y_i^0\in L^2(\mathbb{R}^n)$ ($i=1,2$).  Then,
 there exist positive numbers $\beta=\beta(L, r)\in(0,1)$ and  $C=C(L, r, \|y_i^0\|_{L^2(\mathbb{R}^n)}, T)$  ($i=1,2$)
  such that
\begin{eqnarray}\label{1.5}
\int_{\mathbb{R}^n}|y_1-y_2|^2(x,T)dx\leq C \big(\int_{\omega}|y_1-y_2|^2(x,T)dx\big)^\beta.
\end{eqnarray}
If $y_1(\cdot,T)=y_2(\cdot,T)$ in $\omega$, then
$y_1=y_2$ in $\Omega\times[0,T]$.
\end{theorem}
\begin{remark}\label{remark1.3}
A remark on this theorem  is given below.
{\it
\begin{itemize}
  \item[(d)] Estimate (\ref{1.5}) also implies
   that small changes in the observation lead to small changes in the estimated state.
  It guarantees that solutions of equation (\ref{1.1}) depend continuously on the observation.
  It is regarded as a conditional stability estimate for (\ref{1.1}) in inverse problem (see \cite{Dou, Isakov}).

\end{itemize}}
\end{remark}
 These results provide a framework for estimating the solution from partial observations and have potential applications in control theory and inverse problem. Future work may explore extending these techniques to more general nonlinear parabolic equations.

We organize the paper as follows: In Section 2, some preliminary results are presented.  Section 3 is devoted to the
well-posedness of  the system (\ref{1.1}). In Section 4, we will give some estimates for an auxiliary system.
 In section 5, we will prove our main results.

\section{Well-posedness of  the system (\ref{1.1})}\label{sec3}
\ \ \ \
We start by introducing notation.
We use  $\|\cdot\|_p$ (with  $(p\in[1,+\infty])$)  to denote the norm of
$L^p(\mathbb{R}^n)$.  We use  $C(\ldots)$ to denote a positive constant depending on the enclosed parameters.

We denote by $e^{t\triangle}f$ the semigroup generated by $\triangle$ on $L^q(\mathbb{R}^n)$, where $q\in[1,+\infty)$.
It is well known that
 \begin{eqnarray*}
e^{t\triangle}\varphi:=\int_{\mathbb{R}^n}K_{t}(x-y)\varphi(y)dy, \:\: \textrm{for any} \:\varphi\in L^q(\mathbb{R}^n),
\end{eqnarray*}
where $K_{t}(x):=\frac{1}{t^{n/2}}e^{-\frac{|x|^2}{4t}}\:(t>0)$ is the heat kernel.
 The standard $L^p-L^q$ estimate is provided in Section 1 of \cite{Giga} (see also \cite{Daners}) and we present it as follows.
\begin{lemma}\label{lemma2.1}
If $1\leq q\leq p \leq+\infty$ and $\varphi\in L^q(\mathbb{R}^n)$, then
 \begin{eqnarray}\label{2.1}
\|e^{t\triangle}\varphi\|_{p}\leq (4\pi t)^{-\frac{n}{2}(\frac{1}{q}-\frac{1}{p})}\|\varphi\|_{q},
\;\mbox{ for }\;t>0.
 \end{eqnarray}
\end{lemma}

Using this lemma,  we obtain the estimate for  $\|e^{t\triangle}\|_{L^p\rightarrow L^q}$
 , which is crucial for proving the well-posedness of equation (\ref{1.1}). This leads us to Theorem \ref{theorem3.1}.
\begin{theorem}\label{theorem3.1}
 If $(A_2)$ holds and $y_0\in L^q(\mathbb{R}^n)\cap L^\infty(\mathbb{R}^n)$ ($q\geq1$), then equation (\ref{1.1}) admits  a unique solution  $y\in  L^\infty((0,T^*); L^\infty(\mathbb{R}^n))\cap C([0,T^*]; L^q(\mathbb{R}^n))$ for some positive constant $T^*$.
\end{theorem}

Next, we will discuss this problem further.
While Theorem \ref{theorem3.1} establishes well-posedness under assumption $(A_2$), we now extend our analysis by considering a broader class of initial data. The following theorem addresses the case when $y_0$ only belongs to $L^q(\mathbb{R}^n)$.
\begin{theorem}\label{theorem3.2}
If $(A_3)$ holds, and $y_0\in L^q(\mathbb{R}^n)$, (where $q>\frac{n(p-1)}{2}$, and $q\geq1$) then equation (\ref{1.1}) admits a unique solution  $y\in C([0,T^*]; L^q(\mathbb{R}^n))\cap L_{loc}^\infty((0,T^*); L^\infty(\mathbb{R}^n))$ for some positive constant $T^*$, and $y$
 satisfies
 \begin{eqnarray}\label{3.5}
 \|y(t)\|_{q}+t^{\frac{n}{2q}}\|y(t)\|_\infty\leq C\|y_0\|_{q},\: 0<t\leq T^*,
 \end{eqnarray}
 where the positive constant $C=C(p, q, \|y_0\|_q, T^*)$.
\end{theorem}
To prove Theorem \ref{theorem3.2}, we analyze a related linear parabolic equation with a potential term.
  \begin{equation}\label{2.2}
   \left\{
\begin{array}{ll}
\partial_t u(x,t)-\triangle u(x,t)+a(x,t)u(x,t)=0,& \textrm{ in } \mathbb{R}^{n}\times(0,T],\\
u(x,0)=u_0(x),    &   \textrm{ in } \mathbb{R}^{n}.
\end{array}
\right.
 \end{equation}
The following lemma provides a priori estimate for such equation,  which is crucial in our proof.
\begin{lemma}\label{lemma2.3}
Let $\sigma$ be a positive number with $\sigma>\frac{n}{2}$ and $\sigma\geq1$, and let $a(x,t)\in L^\infty((0,T);L^{\sigma}(\mathbb{R}^n))$. If $u_0\in L^{\gamma}(\mathbb{R}^n)$ $(1\leq \gamma<+\infty)$, then, equation (\ref{2.2}) admits a unique solution
 $u\in C([0,T];L^\gamma(\mathbb{R}^n)\cap L^\infty_{loc}((0,T); L^\infty(\mathbb{R}^n))$. Moreover, there exists
 a positive number $C=C(n,\sigma, \gamma)$  such that for $t\in(0,T]$,
 \begin{eqnarray}\label{2.3a}
\|u(t)\|_{\infty}\leq Ce^{C\mathcal{L}^{\vartheta}t}t^{-\frac{n}{2\gamma}}\|u_0\|_{\gamma},
 \end{eqnarray}
 where
 \begin{eqnarray}\label{2.3}
\mathcal{L}=\|a\|_{L^\infty((0,T);L^{\sigma})} \textrm{ and } \vartheta=\frac{2\sigma}{2\sigma-n}.
 \end{eqnarray}
\end{lemma}
An explicit proof of these theorems is not available in the literature.
 For the readability of this article, we put
 all proofs in Appendix A.

\section{A local interpolation inequality }
\ \ \ \
In this section, we will investigate a local interpolation inequality  for (\ref{1.1}). Throughout this section, we suppose that  $(A_1)$ and $(A_2)$ hold.
Let $\omega$ be an open subset of $\mathbb{R}^n$.
 Let $T>0$, and let $y_i\in L^\infty((0,T);L^\infty(\mathbb{R}^n))\cap C([0,T]; L^2(\mathbb{R}^n))$  be the solution  to equation (\ref{1.1}) corresponding to the initial value $y_i^0\in L^2(\mathbb{R}^n)\cap L^\infty(\mathbb{R}^n)$ $(i=1,2)$.
Write $\varphi=y_1-y_2$. Clearly, $\varphi$ solves the following equation:
 \begin{equation}\label{4.1}
   \left\{
\begin{array}{ll}
\partial_t\varphi-\triangle\varphi+F=0,& \textrm{ in } \mathbb{R}^n\times(0,T];\\
\varphi(x,0)=y^0_1(x)-y^0_2(x),    &   \textrm{ in } \mathbb{R}^n,
\end{array}
\right.
 \end{equation}
where   $F=f(y_1)-f(y_2)$. It follows from  $(A_2)$ that
 \begin{eqnarray}\label{4.2}
|F|\leq L_{M}|\varphi|,
 \end{eqnarray}
where $L_{M}$ is the Lipschitz constant of $f$ on the bounded domain $\mathcal{D}_{M}$ (see Theorem \ref{theorem1}).

We first make a localization process.
By  $(iii)$ of  $(A_1)$, there exist a  number $r>0$ such that for each $j\in \mathbb{N}^{+}$ there is an $x_j\in I_j$ with $B_{r}(x_j)\subset I_j\cap\omega$. \textbf{Let $R=\sqrt{n}L$ in the rest of this article, where $L$ is the side length of the cubes $I_j$.} Clearly, $I_j\subset B_{R}(x_j)$ for each $j\in\mathbb{N}^+$ by $(A_1)$.
We will establish a local interpolation inequality for the solution of (\ref{4.1}).
\begin{theorem}\label{interpolation}
There are two constants $C=C(L, r)>0$ and $\beta=\beta(L, r)\in(0,1)$ such that for each  $j\in \mathbb{N}^+$,
\begin{eqnarray}\label{r4.3}
\int_{I_j}|\varphi(x,T)|^2dx&\leq& e^{C(1+L_M^2)(\frac{1}{T}+T)}\big(\int_{B_{r}(x_j)}|\varphi(x,T)|^2dx\big)^{\beta}\cdot\mathbb{E}_j^{1-\beta},
 \end{eqnarray}
 where $\varphi$ is the solution of (\ref{4.1}), and
  \begin{eqnarray}\label{energy}
\mathbb{E}_j:=\int_{B_{5R}(x_j)}|\varphi(x,0)|^2dx+\int_0^T\int_{B_{5R}(x_j)}|\varphi(x,s)|^2dxds.
\end{eqnarray}
\end{theorem}
In order to prove Theorem \ref{interpolation}, we need several preliminary results. We start with
 the following local energy estimates for the solution of (\ref{4.1}).
 \begin{lemma}\label{lemma4.1}
For each $j\in \mathbb{N}^+$ and $t\in(0,T]$,
  \begin{eqnarray}\label{r4.5}
\int_{B_{3R}(x_j)}|\varphi(x, t)|^2dx\leq L_1(1+L_M)\mathbb{E}_j,
 \end{eqnarray}
 and
  \begin{eqnarray}\label{r4.6}
\int_{B_{3R}(x_j)}|\nabla\varphi(x, t)|^2dx\leq L_1(1+L_M^2)(1+\frac{1}{t})\mathbb{E}_j,
 \end{eqnarray}
 where $L_1>1$ is a constant, which depends only on $R$.
 \end{lemma}
 \begin{proof}
By the method of mollifier, we can find a sequence of smoothing functions $\{\sigma_j(\cdot)\}_{j\in \mathbb{N}^+}\subset C_0^\infty(\mathbb{R}^n)$  satisfying
\begin{eqnarray}\label{truncation}
 0\leq\sigma_j\leq1,\;\;\textrm{supp}\;\sigma_j\subset B_{5R}(x_j),\;\;\sigma_j=1  \;\;\mbox{on}\;\;  B_{4R}(x_j),
\end{eqnarray}
and there exists a positive constant $C=C(R)$ such that for each $j\in\mathbb{N}^+$ and $x\in\mathbb{R}^n$,
\begin{eqnarray}\label{sigma}
|\nabla\sigma_j(x)|\leq C(R).
\end{eqnarray}
Multiplying equation (\ref{4.1}) by $\sigma_j^2\varphi$ and
 integrating it over $\mathbb{R}^n$,  we obtain when $t\in(0,T]$,
\begin{eqnarray*}
\frac{1}{2}\frac{d}{dt}\int_{\mathbb{R}^n}\sigma_j^2|\varphi|^2dx+\int_{\mathbb{R}^n}\nabla\varphi\cdot\nabla(\sigma_j^2\varphi) dx=-\int_{\mathbb{R}^n}\sigma_j^2\varphi Fdx.
\end{eqnarray*}
By direct computations,
\begin{eqnarray*}
\frac{1}{2}\frac{d}{dt}\int_{\mathbb{R}^n}\sigma_j^2|\varphi|^2dx+\int_{\mathbb{R}^n}\sigma_j^2|\nabla\varphi|^2dx
=-2\int_{\mathbb{R}^n}\sigma_j\varphi\nabla\sigma_j\cdot\nabla\varphi dx-\int_{\mathbb{R}^n}\sigma_j^2\varphi Fdx.
\end{eqnarray*}
Applying (\ref{4.2}), (\ref{sigma}), and Cauchy-Schwarz inequality,  we have
\begin{eqnarray}\label{r7.3}
\frac{1}{2}\frac{d}{dt}\int_{\mathbb{R}^n}\sigma_j^2|\varphi|^2dx+\frac{1}{2}\int_{\mathbb{R}^n}\sigma_j^2|\nabla\varphi|^2dx
&\leq& 2\int_{\mathbb{R}^n}|\nabla\sigma_j|^2|\varphi|^2 dx+L_M\int_{\mathbb{R}^n}\sigma_j^2|\varphi|^2dx\nonumber \\
&\leq& C(R)\int_{B_{5R}(x_j)}|\varphi|^2 dx+L_M\int_{B_{5R}(x_j)}|\varphi|^2dx.
\end{eqnarray}
 Integrating (\ref{r7.3}) over $[0,t]$, we  have
\begin{eqnarray}\label{n7.4}
\int_{\mathbb{R}^n}\sigma_j^2|\varphi|^2dx+\int_{0}^{t}\int_{\mathbb{R}^n}\sigma_j^2|\nabla\varphi|^2dxds
\leq C(R)(1+L_M)\mathbb{E}_j,
\end{eqnarray}
from which leads to (\ref{r4.5}).

Then, we construct  another sequence of smoothing functions $\{\tilde{\sigma}_j(\cdot)\}_{j\in \mathbb{N}^+}\subset C_0^\infty(\mathbb{R}^n)$  satisfying
\begin{eqnarray}\label{r7.5}
 0\leq\tilde{\sigma}_j\leq1,\;\;\textrm{supp}\;\tilde{\sigma}_j\subset B_{4R}(x_j),\;  \;\tilde{\sigma}_j=1  \;\;\mbox{on}\;\;  B_{3R}(x_j),
\end{eqnarray}
and there exists a positive constant $C=C(R)$ such that for each $j\in\mathbb{N}^+$ and $x\in\mathbb{R}^n$,
\begin{eqnarray}\label{r7.6}
|\nabla\tilde{\sigma}_j(x)|\leq C(R).
\end{eqnarray}
Multiplying equation (\ref{4.1}) by $t\tilde{\sigma}_j^2\partial_t\varphi$, integrating it over $\mathbb{R}^n$,
and using the integration by parts formula, we obtain
\begin{eqnarray*}
t\int_{\mathbb{R}^n}\tilde{\sigma}_j^2|\partial_t\varphi|^2dx
+\frac{t}{2}\frac{d}{dt}\int_{\mathbb{R}^n}\tilde{\sigma}_j^2|\nabla\varphi|^2 dx
=t\int_{\mathbb{R}^n}2\tilde{\sigma}_j\partial_t\varphi\nabla\tilde{\sigma}_j\cdot\nabla\varphi dx-t\int_{\mathbb{R}^n}\tilde{\sigma}_j^2\partial_t\varphi F dx.
\end{eqnarray*}
It, together with Cauchy-Schwarz inequality, indicates
\begin{eqnarray*}
\frac{t}{2}\frac{d}{dt}\int_{\mathbb{R}^n}\tilde{\sigma}_j^2|\nabla\varphi|^2 dx
&\leq&2t\int_{\mathbb{R}^n}|\nabla\tilde{\sigma}_j|^2|\nabla\varphi|^2dx+\frac{t}{2}\int_{\mathbb{R}^n}\tilde{\sigma}_j^2|F|^2dx.
\end{eqnarray*}
Therefore,
\begin{eqnarray*}
\frac{1}{2}\frac{d}{dt}(t\int_{\mathbb{R}^n}\tilde{\sigma}_j^2|\nabla\varphi|^2dx)
&\leq&\frac{1}{2}\int_{\mathbb{R}^n}\tilde{\sigma}_j^2|\nabla\varphi|^2dx\nonumber \\
&+&C(R)t\int_{B_{4R}(x_j)}|\nabla\varphi|^2dx+\frac{L_M^2t}{2}\int_{B_{4R}(x_j)}|\varphi|^2dx.
\end{eqnarray*}
Integrating it over $[0,t]$, we have
\begin{eqnarray}\label{r7.7}
t\int_{\mathbb{R}^n}\tilde{\sigma}_j^2|\nabla\varphi|^2dx&\leq&(1+2C(R)t)\int_0^t\int_{B_{4R}(x_j)}|\nabla\varphi|^2dxds  \nonumber\\
&+&L_M^2t\int_0^t\int_{B_{4R}(x_j)}|\varphi|^2dxds.
\end{eqnarray}
Combining (\ref{r7.7}) with  (\ref{n7.4}) leads to
  \begin{eqnarray*}
\int_{\mathbb{R}^n}\tilde{\sigma}_j^2|\nabla\varphi|^2dx\leq C(R)(1+L_M^2)(1+\frac{1}{t})\mathbb{E}_j.
 \end{eqnarray*}
It, along with  (\ref{r7.5}), indicates (\ref{r4.6}).
Hence, we complete the proof.
\end{proof}

Then, we have
\begin{lemma}\label{theta}
If $\int_{B_{R}(x_j)}|\varphi(x,T)|^2dx\neq0$ ($j\in \mathbb{N}^+$), then,
 there are  positive constants $L_2, L_3$, $L_4$, and $L_5$ (depending only on $R$) such that
\begin{eqnarray}\label{n4.7}
\mathbb{E}_j\leq e^{\frac{L_2}{\theta_j}}\int_{B_{2R}(x_j)}|\varphi(x,t)|^2dx,\:\textrm{ for }\:  T-\theta_j\leq t\leq T,
\end{eqnarray}
where $\mathbb{E}_j$ is given in (\ref{energy}) and
\begin{eqnarray}\label{n4.8}
\frac{1}{\theta_j}:=L_3\ln\bigg( e^{L_4L_{M}T+L_5(1+\frac{1}{T})}\frac{\mathbb{E}_j}{\int_{B_{R}(x_j)}|\varphi(x,T)|^2dx}\bigg),
\end{eqnarray}
with $0<\theta_j<\min\{1,{T}/{2}\}$.
\end{lemma}
\begin{remark}\label{remark4.1}
Inequality (\ref{n4.8}) shows  that if $\int_{B_{R}(x_j)}|\varphi(x,T)|^2dx\neq0$, then $\int_{B_{2R}(x_j)}|\varphi(x,t)|^2dx\neq0$, when $T-\theta_j\leq t\leq T$. The proof of this lemma can be found in \cite{G. Zheng} and we omit the details.
\end{remark}
Next, we will introduce the definition of the frequency function. To this end, we first construct a sequence of smoothing functions $\{\eta_j(\cdot)\}_{j\in \mathbb{N}^+}\subset C_0^\infty(\mathbb{R}^n)$ satisfying
\begin{eqnarray}\label{r4.11}
 0\leq\eta_j\leq1, \: \textrm{supp}\; \eta_j\subset  B_{3R}(x_j), \;\; \eta_j=1  \textrm{ on }  B_{5R/2}(x_j),
\end{eqnarray}
and there exits a positive number $C=C(R)>0$ such that for any $x\in \mathbb{R}^n$ and $j\in \mathbb{N}^+$,
\begin{eqnarray}\label{r4.12}
 |\nabla\eta_j(x)|\leq C(R).
 \end{eqnarray}
Define $\phi_j(x,t)=\eta_j(x)\varphi(x,t)$.
 For each  $h>0$, we define functions
\begin{eqnarray}\label{G}
G_{h,j}(x,t):=\frac{1}{(T-t+h)^{n/2}}e^{-\frac{|x-x_j|^2}{4(T-t+h)}},\:\: (x,t)\in \mathbb{R}^n\times[0,T]
\end{eqnarray}
and
\begin{eqnarray}\label{N}
N_{h,j}(t):=\frac{\int_{B_{3R}(x_j)}|\nabla \phi_j|^2(x,t)\cdot G_{h,j}(x,t)dx}{\int_{B_{3R}(x_j)}|\phi_j|^2(x,t)\cdot G_{h,j}(x,t)dx},
\end{eqnarray}
when $t\in (0,T]$ and $\int_{B_{3R}(x_j)}|\phi_j|^2(x,t)\cdot G_{h,j}(x,t)dx\neq0$.
\begin{remark}
We call the function $N_{h,j}(\cdot)$ the local
{\it  frequency function}, and  it  has the following properties.
\end{remark}

\begin{lemma}\label{lemma2.4}
$(i)$ If $t\in(0,T]$,  $h>0$, and $\int_{B_{3R}(x_j)}|\phi_j|^2(x,t)\cdot G_{h,j}(x,t)dx\neq0$, then
\begin{eqnarray*}\label{4.3}
&&\frac{1}{2}\frac{d}{dt}\int_{ B_{3R}(x_j)}|\phi_j|^2(x,t)\cdot G_{h,j}(x,t)dx+N_{h,j}(t)\int_{B_{3R}(x_j)}|\phi_j|^2(x,t)\cdot G_{h,j}(x,t)dx\nonumber \\
&=&\int_{B_{3R}(x_j)} \!\phi_j\cdot(\partial_t \phi_j-\triangle \phi_j)(x,t)\cdot G_{h,j}(x,t)dx.
\end{eqnarray*}
$(ii)$
If $t\in(0,T]$,  $h>0$, and $\int_{B_{3R}(x_j)}|\phi_j|^2(x,t)\cdot G_{h,j}(x,t)dx\neq0$, then
\begin{eqnarray*}\label{4.4}
N_{h,j}' (t)\leq\frac{1}{T-t+h}N_{h,j}(t)+\frac{\int_{B_{3R}(x_j)}(\partial_t \phi_j-\triangle \phi_j)^2(x,t)\cdot G_{h,j}(x,t)dx}{\int_{B_{3R}(x_j)}|\phi_j|^2(x,t)\cdot G_{h,j}(x,t)dx}.
\end{eqnarray*}
\end{lemma}
\begin{proof}
The proof can be found in \cite{PW} and we omit the details.
\end{proof}
 Next, we will introduce an interpolation inequality.
\begin{lemma}\label{lemma2.5}
Let $h>0$, let $0<S<T$, and let $g(\cdot), N(\cdot)\in C^1[S,T]$ be two positive functions satisfying
  \begin{eqnarray*}
   \left\{
\begin{array}{ll}
\big|\frac{1}{2}g'(t)+N(t)g(t)\big|\leq \tilde{C}g(t), \\
N'(t)\leq\frac{1}{ T-t+h}N_{\lambda}(t)+\bar{C},
\end{array}
\right.
 \end{eqnarray*}
 where $\tilde{C}$ and $\bar{C}$ are two positive constants. Then for $S\leq t_1<t_2<t_3\leq T$, we have
 \begin{eqnarray}\label{logcon}
 g(t_2)^{1+D}\leq g(t_3)\cdot g(t_1)^{D}\cdot e^{\mathcal{K}},
 \end{eqnarray}
where
 \begin{eqnarray*}
 D=\frac{\int_{t_2}^{t_3}\frac{dt}{T-t+h}}{\int_{t_1}^{t_2}\frac{dt}{T-t+h}},
 \end{eqnarray*}
 and
  \begin{eqnarray*}
  \mathcal{K}=2D(t_2-t_1)\big[\tilde{C}+\bar{C}(t_2-t_1)\big]+2(t_3-t_2)\big[\tilde{C}+\bar{C}\int_{t_2}^{t_3}\frac{T-t_2+h}{T-t+h}dt\big].
 \end{eqnarray*}
\end{lemma}
\begin{remark}
This inequality is borrowed from \cite{Phung}, and we also omit the details.
Here, we call it the logarithmic convexity property of the function $g(\cdot)$.
\end{remark}

Now, we are in the position to prove Theorem \ref{interpolation}.
\begin{proof}
 Without loss of generality,
we assume that  $\int_{I_j}|\varphi(x,T)|^2dx\neq0$ ($j\in \mathbb{N}^+$).
We recall $R=\sqrt{n}L$, and $I_j\subset B_{R}(x_j)$  for each $j\in\mathbb{N}^+$.
Thus,
\begin{eqnarray}\label{r4.9}
\int_{B_{R}(x_j)}|\varphi(x,T)|^2dx\neq0.
\end{eqnarray}
Now,
we are going to prove the following inequalities:
there exist two constants $C=C(R, r)>0$ and $\beta=\beta(R, r)\in(0,1)$ such that for each  $j\in \mathbb{N}^+$,
\begin{eqnarray}\label{r4.10}
\int_{B_{R}(x_j)}|\varphi(x,T)|^2dx&\leq& e^{C(1+L_M^2)(\frac{1}{T}+T)}\big(\int_{B_{r}(x_j)}|\varphi(x,T)|^2dx\big)^{\beta}\cdot\mathbb{E}_j^{1-\beta}.
 \end{eqnarray}
When it is proved, we finish the proof of this theorem.
The proof of (\ref{r4.10}) will be split into three steps.\

\noindent{\it Step 1. We study the properties of the frequency function $N_{h,j}(t)$.}\

We recall that the functions $\phi_j(x,t)=\eta_j(x)\varphi(x,t)$, where $\eta_j(x)$ is provided in (\ref{r4.11}).
Using (\ref{r4.9}) and Remark \ref{remark4.1} , we have $\int_{B_{3R}(x_j)}|\phi_j(x,t)|^2\cdot G_{h,j}(x,t)dx\neq0$, when $t\in[T-\theta_j,T]$  (where $\theta_j$ is the positive number given in Lemma \ref{theta}).
Therefore, the frequency function $N_{h,j}(t)$ given in (\ref{N}) is well-defined,  when $t\in[T-\theta_j,T]$.
After some computations,
 \begin{equation}\label{r4.15}
\partial_t\phi_j-\triangle\phi_j=H_j-\eta_j F,\:\: \textrm{ in } \mathbb{R}^n\times(0,T],
 \end{equation}
where $H_j=-2\nabla\varphi\cdot\nabla\eta_j-\varphi\triangle\eta_j$.
This, along with $(i)$ of Lemma \ref{lemma2.4}, indicates
\begin{eqnarray*}
&&\frac{1}{2}\frac{d}{dt}\int_{B_{3R}(x_j)}|\phi_j|^2\cdot G_{h,j} dx+N_{h,j}(t)\int_{B_{3R}(x_j)}|\phi_j|^2\cdot G_{h,j}dx\\
&=&\int_{B_{3R}(x_j)} \phi_j(H_j-\eta_j F)\cdot G_{h,j}dx.
\end{eqnarray*}
It, together with Cauchy-Schwarz inequality,  indicates when $t\in[T-\theta_j,T]$
\begin{eqnarray}\label{r4.16}
&&|\frac{1}{2}\frac{d}{dt}\int_{B_{3R}(x_j)}|\phi_j|^2\cdot G_{h,j} dx+N_{h,j}(t)\int_{B_{3R}(x_j)}|\phi_j|^2\cdot G_{h,j} dx|\nonumber \\
&\leq&\frac{1}{2}\int_{B_{3R}(x_j)}\big(|H_j|^2+\eta_j^2|F|^2\big)G_{h,j} dx+\int_{B_{3R}(x_j)} |\phi_j|^2\cdot G_{h,j}dx\\
&=&\big(\frac{1}{2}\frac{\int_{B_{3R}(x_j)}\big(|H_j|^2+\eta_j^2|F|^2\big)G_{h,j} dx}{\int_{B_{3R}(x_j)} |\phi_j|^2\cdot G_{h,j}dx}+1\big)\int_{B_{3R}(x_j)} |\phi_j|^2\cdot G_{h,j}dx.\nonumber
\end{eqnarray}
Combining $(ii)$ of Lemma \ref{lemma2.4}, (\ref{r4.15}), and Cauchy-Schwarz inequality leads to
when $t\in[T-\theta_j,T]$
\begin{eqnarray}\label{r4.17}
\frac{dN_{h,j} }{dt}(t)\leq{1\over T-t+h}N_{h,j}(t)+\frac{2\int_{B_{3R}(x_j)}\big(|H_j|^2+\eta_j^2|F|^2\big)G_{h,j}dx}{\int_{B_{3R}(x_j)}|\phi_j|^2G_{h,j}dx}.
\end{eqnarray}

\noindent{\it Step 2.We will estimate $\frac{\int_{B_{3R}(x_j)}(|H_j|^2+\eta_j^2|F|^2)G_{h,j}dx}{\int_{B_{3R}(x_j)}|\phi_j|^2\cdot G_{h,j} dx}$, when $t\in[T-\varepsilon_j,T]\subset[T-\theta_j,T]$,
where the parameter $\varepsilon_j\in(0,\theta_j)$ ($j\in \mathbb{N}^+$) will be given later.}\

Clearly, we have
\begin{eqnarray}\label{r4.19}
H_j(x,t)=0, \: \textrm{for} \:x\in B_{5R/2}(x_j), \: t\in[0, T].
\end{eqnarray}
This  indicates
\begin{eqnarray*}
\frac{\int_{B_{3R}(x_j)} |H_j|^2\cdot G_{h,j} dx}{\int_{B_{3R}(x_j)}|\phi_j|^2\cdot G_{h,j}dx}&\leq & \frac{\int_{B_{3R}(x_j)\setminus B_{5R/2}(x_j)}|H_j|^2\cdot G_{h,j}dx}{\int_{B_{2R}(x_j)}|\phi_j|^2\cdot G_{h,j} dx}\\
&\leq& \frac{\int_{B_{3R}(x_j)\setminus B_{5R/2}(x_j)}|H_j|^2dx}{\int_{B_{2R}(x_j)}|\phi_j|^2dx}e^{-\frac{L_6}{T-t+h}},
\:t\in[T-\theta_j,T],
\end{eqnarray*}
where $L_6:=\frac{9R^2}{16}$.
Using  (\ref{r4.12}) and Cauchy-Schwarz inequality,
\begin{eqnarray}\label{r4.20}
\int_{B_{3R}(x_j)} |H_j(x,t)|^2dx
&\leq&\int_{B_{3R}(x_j)} (|\nabla\varphi|^2+|\varphi|^2)(x,t)\cdot(4|\nabla\eta_j|^2+|\triangle\eta_j|^2)(x,t)dx \\
&\leq& C_1\int_{B_{3R}(x_j)} (|\nabla\varphi|^2+|\varphi|^2)(x,t)dx,\nonumber
\end{eqnarray}
where $C_1$ is a positive number only depending on $R$.
This, along with  (\ref{r4.5}) and  (\ref{r4.6}), implies that
\begin{eqnarray}\label{r4.21}
&&\int_{B_{3R}(x_j)} |H_j(x,t)|^2dx\leq C_1L_1(1+L_M^2)(1+\frac{1}{t})\mathbb{E}_j,
\end{eqnarray}
where $L_1>1$ is the constant given in Lemma \ref{lemma4.1}. Therefore,
\begin{eqnarray*}
\frac{\int_{B_{3R}(x_j)} |H_j|^2\cdot G_{h,j} dx}{\int_{B_{3R}(x_j)}|\phi_j|^2\cdot G_{h,j} dx}\leq  \frac{C_1L_1(1+L_M^2)(1+t^{-1})\mathbb{E}_j}{\int_{B_{2R}(x_j)}|\phi_j|^2dx}e^{-\frac{L_6}{T-t+h}}.
\end{eqnarray*}
By Lemma \ref{theta}, we observe when $t\in[T-\theta_j,T]$,
 \begin{eqnarray}\label{r4.18}
t^{-1}\leq\frac{2}{T}.
\end{eqnarray}
These, along with Lemma \ref{theta}, indicate when $t\in[T-\varepsilon_j,T]\subset[T-\theta_j,T]$,
\begin{eqnarray}\label{r4.22}
\frac{\int_{B_{3R}(x_j)} |H_j|^2\cdot G_{h,j} dx}{\int_{B_{3R}(x_j)}|\phi_j|^2\cdot G_{h,j} dx}\leq 2C_1L_1(1+L_M^2)(1+\frac{1}{T})e^{\frac{L_2}{\theta_j}}e^{-\frac{L_6}{\varepsilon_j+h}}.
\end{eqnarray}
Let
\begin{eqnarray}\label{r4.23}
\varepsilon_j=k\theta_j \textrm{ and } h=\mu\varepsilon_j,
\end{eqnarray}
 where  the parameter
  $\mu\in(0,1)$ will be given later and
 \begin{eqnarray}\label{re1}
k:=\min\{\frac{L_6}{2L_2},\frac{1}{2}\}
\end{eqnarray}
It is obvious that  $L_2-\frac{L_6}{k(1+\mu)}<0$.
These, together with (\ref{r4.23}), yield
\begin{eqnarray*}\label{5.11}
\frac{\int_{B_{3R}(x_j)} |H_j|^2\cdot G_{h,j} dx}{\int_{B_{3R}(x_j)}|\phi_j|^2\cdot G_{h,j} dx}&\leq& 2C_1L_1(1+L_M^2)(1+\frac{1}{T})e^{\frac{L_2}{\theta_j}}e^{-\frac{L_6}{k(1+\mu)\theta_j}}\\
&\leq& 2C_1L_1(1+L_M^2)(1+\frac{1}{T}).
\end{eqnarray*}
Then by (\ref{4.2}), we have
\begin{eqnarray*}
\frac{\int_{B_{3R}(x_j)}\eta_j^2|F|^2G_{h,j}dx}{\int_{B_{3R}(x_j)}|\phi_j|^2G_{h,j} dx}\leq L_{M}^2.
\end{eqnarray*}
In summary, we obtain when $t\in[T-\varepsilon_j,T]$,
\begin{eqnarray}\label{r4.25}
\frac{\int_{B_{3R}(x_j)}\big(|H_j|^2+\eta_j^2|F|^2\big)G_{h,j}dx}{\int_{B_{3R}(x_j)}|\phi_j|^2G_{h,j}dx}\leq 2(1+C_1)L_1(1+L_M^2)(1+\frac{1}{T}).
\end{eqnarray}

\noindent{\it Step 3. We prove inequality (\ref{r4.10}).}\\

By (\ref{r4.16}), (\ref{r4.17}), and (\ref{r4.25}), we have when $t\in[T-\varepsilon_j,T]$
\begin{eqnarray*}
|\frac{1}{2}\frac{d}{dt}\int_{B_{3R}(x_j)}|\phi_j|^2\cdot G_{h,j} dx&+&N_{h,j}(t)\int_{B_{3R}(x_j)}|\phi_j|^2\cdot G_{h,j} dx|\nonumber \\
&\leq&\big((1+C_1)L_1(1+L_M^2)(1+\frac{1}{T})+1\big)\int_{B_{3R}(x_j)} |\phi_j|^2G_{h,j}dx
\end{eqnarray*}
and
\begin{eqnarray*}
\frac{dN_{h,j}}{dt}(t)\leq
\frac{1}{T-t+h_j}N_{h,j}(t)+4(1+C_1)L_1(1+L_M^2)(1+\frac{1}{T}).
\end{eqnarray*}
Write $f_j(t)=\int_{B_{3R}(x_j)}|\phi_j(x,t)|^2\cdot G_{h,j}(x,t)dx$.
Then,  the following differential inequalities hold:
 \begin{eqnarray}\label{inequality}
   \left\{
\begin{array}{ll}
|\frac{1}{2}f_j'(t)+N_{h,j}(t)f_j(t)|
\leq\big(\hat{C}+1\big)f_j(t);\\
N_{h,j}'(t)\leq
\frac{1}{ T-t+\lambda}N_{h,j}(t)+4\hat{C},
\end{array}
\right.
 \end{eqnarray}
 where $t\in[T-\varepsilon_j, T]$ and
 \begin{eqnarray}\label{constant}
\hat{C}=(1+C_1)L_1(1+L_M^2)(1+\frac{1}{T}).
 \end{eqnarray}

Take a positive number $l$  such that
\begin{eqnarray}\label{r4.27}
r^2(1+l)=18R^2(1+D_l),
 \end{eqnarray}
where
\begin{eqnarray}\label{re2}
D_l=\frac{\ln(l+1)}{\ln(\frac{2l+1}{l+1})}.
\end{eqnarray}
It is easy to check the number $l$ only depends on $(R,r)$ and $l>1$.
We choose the parameter $\mu$ in (\ref{r4.23}) such that
  \begin{eqnarray}\label{r4.29}
 \frac{r^2}{8h}=\ln2&+&2(1+D_l)\ln(L_1(1+L_M))+\frac{(1+D_l)}{\theta_j}(L_2+\frac{18R^2}{k})+\mathcal{K}_{l}\nonumber \\
 &+&(1+D_l)\ln\big(e^{L_4L_{M}T+L_5(1+\frac{1}{T})}\frac{\mathbb{E}_j}{\int_{B_R(x_j)}|\varphi(x,T)|^2dx}\big),
  \end{eqnarray}
   where constants $L_2, L_4, L_5$ are given in Lemma \ref{theta} and
 \begin{eqnarray}\label{rr4.31}
  \mathcal{K}_{l}=2D_l\big[(\hat{C}+1)lh+4\hat{C}(lh)^2\big]+2(\hat{C}+1)lh+8\hat{C}l(l+1)h^{2}\cdot\ln(l+1).
 \end{eqnarray}
 We first observe
  \begin{eqnarray*}
 \frac{r^2}{8h}>\frac{(1+D_l)}{\theta_j}\frac{18R^2}{k}.
 \end{eqnarray*}
This, along with (\ref{r4.23}) and (\ref{r4.27}) shows that $\mu\in(0,1)$. Moreover, we have
   \begin{eqnarray}\label{r4.31}
0<2hl<\frac{r^2l}{4\frac{(1+D_l)}{\theta_j}\frac{18R^2}{k}}<\varepsilon_j.
  \end{eqnarray}
 Then, we choose $t_3=T$, $t_2=T-lh$, and $t_1=T-2lh$. Clearly, $T-\varepsilon_j<t_1<t_2<t_3=T$.
Applying Lemma \ref{lemma2.5} and (\ref{inequality}), we obtain
\begin{eqnarray}\label{5.13}
f_j(t_2)^{1+D_l}\leq f_j(t_3)\cdot f_j(t_1)^{D_l}\cdot e^{\mathcal{K}_{l}}.
 \end{eqnarray}


Using (\ref{r4.5}) and (\ref{r4.11}), we have
\begin{eqnarray*}
f_j(t_1)=\int_{ B_{3R}(x_j)}|\phi_j(x,t_1)|^2\cdot e^{-\frac{|x-x_j|^2}{4h(2l+1)}}dx\leq L_1(1+L_M)\mathbb{E}_j;
 \end{eqnarray*}
 \begin{eqnarray*}
f_j(t_3)&=&\int_{B_{3R}(x_j)}|\phi(x,T)|^2\cdot e^{-\frac{|x-x_j|^2}{4h}}dx\\
&\leq&\int_{B_{r}(x_j)}|\varphi(x,T)|^2dx
+e^{-\frac{r^2}{4h}} L_1(1+L_M)\mathbb{E}_j.
 \end{eqnarray*}
These, together with (\ref{5.13}), indicate
\begin{eqnarray*}
\big[\int_{B_{3R}(x_j)}|\phi(x,t_2)|^2dx\big]^{1+D_l}&\leq& e^{\frac{9R^2(1+D_l)}{4h(l+1)}}\cdot e^{\mathcal{K}_{l}}\cdot L_1(1+L_M)\\
&\times&\big(\int_{B_{r}(x_j)}|\varphi(x,T)|^2dx
+e^{-\frac{r^2}{4h}}\mathbb{E}_j\big)\cdot[L_1(1+L_M)]^{D_l}\cdot\mathbb{E}_j^{D_l}.
 \end{eqnarray*}
It follows from (\ref{r4.27}) that
\begin{eqnarray}\label{5.16}
\big[\int_{B_{3R}(x_j)}|\phi(x,t_2)|^2dx\big]^{1+D_l}&\leq & e^{\mathcal{K}_{l}}\cdot [L_1(1+L_M)]^{1+D_l}\nonumber \\
&\times&\big(e^{\frac{r^2}{8h}}\int_{B_{r}(x_j)}|\varphi(x,T)|^2dx
+e^{-\frac{r^2}{8h}}\mathbb{E}_j\big)\cdot\mathbb{E}_j^{D_l}.
 \end{eqnarray}
 Using (\ref{r4.5}) and Lemma \ref{theta}, we obtain
\begin{eqnarray*}
\int_{B_R(x_j)}|\varphi(x,T)|^2dx&\leq&  L_1(1+L_M)\cdot\mathbb{E}_j\\
&\leq & L_1(1+L_M)\cdot e^{\frac{1}{\theta_j}(L_2+\frac{18R^2}{k})}\cdot\int_{B_{2R}(x_j)}|\varphi(x,t_2)|^2dx.
 \end{eqnarray*}
This, together with (\ref{5.16}) and (\ref{r4.11}), yields
\begin{eqnarray}\label{r4.34}
\big[\int_{B_{R}(x_j)}|\varphi(x,T)|^2dx\big]^{1+D_l}&\leq &  [L_1(1+L_M)]^{2(1+D_l)}\cdot  e^{\frac{(1+D_l)}{\theta_j}(L_2+\frac{18R^2}{k})}\cdot e^{\mathcal{K}_{l}}\nonumber \\
&\times&\!\!\!\!\!\big(e^{\frac{r^2}{8h}}\int_{B_{r}(x_j)}|\varphi(x,T)|^2dx
+e^{-\frac{r^2}{8h}}\mathbb{E}_j\big)\cdot\mathbb{E}_j^{D_l}.
 \end{eqnarray}
By (\ref{r4.29}), we have
  \begin{eqnarray*}
e^{-\frac{r^2}{8h}}\cdot  [L_1(1+L_M)]^{2(1+D_l)}\!\cdot\!e^{\mathcal{K}_{l}}\!\cdot\!  e^{\frac{(1+D_l)}{\theta_j}(L_2+\frac{18R^2}{k})}\!\cdot\!\mathbb{E}_j^{1+D_l}\leq\frac{1}{2}\big(\int_{B_R(x_j)}|\varphi(x,T)|^2dx\big)^{1+D_l}.
 \end{eqnarray*}
It, together with (\ref{r4.34}), indicates
\begin{eqnarray}\label{r4.35}
\big[\int_{B_R(x_j)}|\varphi(x,T)|^2dx\big]^{1+D_l}&\leq& 2[L_1(1+L_M)]^{2(1+D_l)}\cdot e^{\frac{(1+D_l)}{\theta_j}(L_2+\frac{18R^2}{k})}\cdot
e^{\mathcal{K}_{l}} \nonumber \\
&\times&\big(e^{\frac{r^2}{8h}}\int_{B_{r}(x_j)}|\varphi(x,T)|^2dx
)\cdot\mathbb{E}_j^{D_l}.
 \end{eqnarray}
 Using (\ref{r4.29}) again, we obtain
\begin{eqnarray}\label{r4.36}
 e^{\frac{r^2}{8h}}&=&2[L_1(1+L_M)]^{2(1+D_l)}\cdot}e^{\frac{(1+D_l)}{\theta_j}(L_2+\frac{18R^2}{k})}\cdot e^{\mathcal{K}_{l}\nonumber \\
 &\times&\big(e^{L_4L_{M}T+L_5(1+\frac{1}{T})}\frac{\mathbb{E}_j}{\int_{B_R(x_j)}|\varphi(x,T)|^2dx}\big)^{1+D_l},
\end{eqnarray}
By Lemma \ref{theta},
 \begin{eqnarray*}
e^{\frac{(1+D_l)}{\theta_j}(L_2+\frac{18R^2}{k})}=
\bigg(e^{L_4L_{M}T+L_5(1+\frac{1}{T})}\frac{\mathbb{E}_j}{\int_{B_{R}(x_j)}|\varphi(x,T)|^2dx}\bigg)^{L_3(1+D_l)(L_2+\frac{18R^2}{k})}.
\end{eqnarray*}
It, together with (\ref{r4.35}) and (\ref{r4.36}), deduces
\begin{eqnarray*}
&&\big[\int_{B_R(x_j)}|\varphi(x,T)|^2dx\big]^{1+D_l}\\
&\leq& \mathcal{L}\big (\frac{\mathbb{E}_j}{\int_{B_R(x_j)}|\varphi(x,T)|^2dx}\big)^{(1+D_l)[1+2L_3(L_2+\frac{18R^2}{k})]}\int_{B_{r}(x_j)}|\varphi(x,T)|^2dx
\cdot\mathbb{E}_j^{D_l},
 \end{eqnarray*}
where
\begin{eqnarray}\label{r4.37}
\mathcal{L}=4[L_1(1+L_M)]^{4(1+D_l)}\cdot e^{2\mathcal{K}_{l}} \cdot e^{[L_4L_{M}T+L_5(1+\frac{1}{T})]\cdot(1+D_l)\cdot[1+2L_3(L_2+\frac{18R^2}{k})]}.
 \end{eqnarray}
It follows from (\ref{constant}), (\ref{rr4.31}), $l>1$ and  $2lh<\theta_j<\min\{1, T/2\}$ that
 \begin{eqnarray}\label{r4.39}
 e^{\mathcal{K}_{l}}\leq e^{C_2(1+L_M^2)(1+\frac{1}{T})},
 \end{eqnarray}
where $C_2$ is a positive number only depending on $(R, r)$.
Therefore,
\begin{eqnarray}\label{r4.40}
\int_{B_R(x_j)}|\varphi(x,T)|^2dx\leq \mathcal{L}^{\beta}\big(\int_{B_{r}(x_j)}|\varphi(x,T)|^2dx\big)^{\beta}
\cdot\mathbb{E}_j^{1-\beta},
 \end{eqnarray}
where
\begin{eqnarray}\label{r4.41}
\beta=\frac{1}{2(1+D_l)\cdot(1+L_3(L_2+\frac{18R^2}{k}))}.
 \end{eqnarray}
It follows from (\ref{re1}), (\ref{re2}), and Lemma \ref{theta} that $\beta$ only depends on $(R,r)$.
We obtain (\ref{r4.10}). This  complete the proof.
\end{proof}

\section{Proof of the main results}
\ \ \ \
In this section, we present the proofs of Theorem \ref{theorem1} and Theorem \ref{theorem2} using the techniques developed in the previous sections.

\subsection{Proof of Theorem \ref{theorem1}}
\begin{proof}
The proof will be organized in two steps.

\noindent{\it Step 1.   We  prove (\ref{1.4}).}\

Together Theorem \ref{interpolation} with Young's inequality, we have that for each $j\in \mathbb{N}^+$,
\begin{eqnarray}\label{r5.1}
\int_{I_j}|\varphi(x,T)|^2dx&\leq& e^{C(1+L_M^2)(\frac{1}{T}+T)}[\epsilon\int_{B_{r}(x_j)}|\varphi(x,T)|^2dx+C(\epsilon)\mathbb{E}_j],
 \end{eqnarray}
where $\epsilon>0$ will be given later and $C(\epsilon)=(1-\beta)(\frac{\epsilon}{\beta})^{-\frac{\beta}{1-\beta}}$.
Summing (\ref{r5.1}) for $j\in \mathbb{N}^+$ and using $(iii)$ of $(A_1)$, we deduce
\begin{eqnarray}\label{r5.2}
\int_{R^{n}}|\varphi(x,T)|^2dx&\leq& e^{C(1+L_M^2)(\frac{1}{T}+T)}\big(\epsilon\int_{\omega}|\varphi(x,T)|^2dx+C(\epsilon)\sum_{j\in\mathbb{N}^+}\mathbb{E}_j\big).
 \end{eqnarray}
Clearly,  there are only finite $j\in \mathbb{N}^+$ satisfying  $x\in B_{5R}(x_j)$ for each $x\in\mathbb{R}^n$, and  there exists a positive number $C_3>1$ (which only depends on $n$) such that for each $t\in[0,T]$,
\begin{eqnarray*}
\sum_{j\in \mathbb{N}^+}\int_{B_{5R}(x_j)}|\varphi(x,t)|^2dx\leq C_3\int_{\mathbb{R}^n}|\varphi(x,t)|^2dx.
\end{eqnarray*}
Therefore, we have
\begin{eqnarray*}
\sum_{j\in\mathbb{N}^+}\mathbb{E}_j\leq C_3\mathbb{E},
\end{eqnarray*}
where
\begin{eqnarray}\label{global}
\mathbb{E}:=\int_{R^{n}}|\varphi(x,0)|^2dx+\int_0^T\int_{R^{n}}|\varphi(x,s)|^2dxds.
\end{eqnarray}
It, along with (\ref{r5.2}), indicates that
\begin{eqnarray}\label{r5.4}
\int_{R^{n}}|\varphi(x,T)|^2dx\leq C_3e^{C(1+L_M^2)(\frac{1}{T}+T)}\big(\epsilon\int_{\omega}|\varphi(x,T)|^2dx+C(\epsilon)\mathbb{E}\big).
 \end{eqnarray}
Taking $\epsilon=\big(\frac{\mathbb{E}}{\int_{\omega}|\varphi(x,T)|^2dx}\big)^{1-\beta}$ in (\ref{r5.4}), we obtain
\begin{eqnarray}\label{r5.5}
\int_{R^{n}}|\varphi(x,T)|^2dx\leq C_3e^{C(1+L_M^2)(\frac{1}{T}+T)}\big(1+(1-\beta)\beta^{\frac{\beta}{1-\beta}}\big)
\big(\int_{\omega}|\varphi(x,T)|^2dx\big)^{\beta}\mathbb{E}^{1-\beta}.
 \end{eqnarray}
By a standard energy method, we have for each $t\in[0,T]$,
\begin{eqnarray*}
\int_{R^{n}}|\varphi(x,t)|^2dx\leq e^{2L_Mt}\int_{R^{n}}|\varphi(x,0)|^2dx.
 \end{eqnarray*}
Thus,
\begin{eqnarray}\label{r5.6}
\mathbb{E}\leq(1+Te^{2L_MT})\int_{R^{n}}|\varphi(x,0)|^2dx.
 \end{eqnarray}
Clearly,  $1+ T\cdot e^{2L_{M}T}\leq(1+T)e^{2L_{M}T}\leq e^T\cdot e^{2L_{M}T}$.
Combining the above, (\ref{r5.5}), and  (\ref{r5.6}) leads to (\ref{1.3}).\

\noindent{\it Step 2.   We  prove \eqref{1.4}.}

 We first claim that   $\varphi(\cdot,t)\neq0$ for each $t\in(0,T]$.
By contradiction, we assume that  $\varphi(\cdot,t)=0$ for some $t\in(0,T]$.
Since  $\varphi(\cdot,0)\neq0$, there exists a $t_0\in(0, T]$ satisfying
 \begin{eqnarray*}
t_0:=\inf\{t\in(0,T]\:|\: \varphi(\cdot,t)=0\}.
 \end{eqnarray*}
 Since  $\varphi\in C([0,T]; L^2(\mathbb{R}^n))$, we have
  \begin{eqnarray}\label{r5.7}
\varphi(\cdot, t_0)=0  \textrm{  and  }  \varphi(\cdot, t)\neq0,  \textrm{  for each  } t\in[0,t_0).
 \end{eqnarray}
Define a function
   \begin{eqnarray}\label{r5.8}
\chi(t) := {\|\varphi(\cdot, t)\|_2^2\over \|\varphi(\cdot, t)\|^2_{H^{-1}}},\; t\in[0,t_0),
 \end{eqnarray}
 where  $\|\cdot\|_{H^{-1}}$ is the norm in  $H^{-1}(\mathbb{R}^n)$.

By direct computations,  we obtain that for $t\in(0,T]$,
\begin{eqnarray}\label{r5.9}
\begin{cases}
{1\over 2}{d\over dt}\|\varphi(\cdot, t)\|_2^2 +\|\varphi(\cdot, t)|\|_{H^1}^2 = \langle -F(\cdot, t)+\varphi(\cdot, t), \varphi(\cdot, t)\rangle,\\
{1\over 2}{d\over dt}\|\varphi(\cdot, t)\|_{H^{-1}}^2 + \|\varphi(\cdot, t)\|_2^2 = \langle -F(\cdot, t)+\varphi(\cdot, t),\varphi(\cdot, t)\rangle_{H^{-1}},
\end{cases}
\end{eqnarray}
where $\langle \cdot,\cdot\rangle$  is the inner product of $L^2(\R^n)$,
 $\langle\cdot,\cdot\rangle_{H^{-1}}$ stands for  the inner product of $H^{-1}(\R^n)$, and
 $\|\cdot\|_{H^{1}}$ is the norm in $H^{1}(\R^n)$.
Therefore,
\begin{eqnarray}\label{r5.10}
\chi'(t)
=\frac{2}{\|\varphi\|_{H^{-1}}^4}&\times&\big(-\langle F, \varphi\rangle\|\varphi\|^2_{H^{-1}}-\|\varphi\|^2_{H^1}\|\varphi\|^2_{H^{-1}}\nonumber \\
&+&\langle F, \varphi\rangle_{H^{-1}}\|\varphi\|_2^2 +\|\varphi\|_2^4\big), \;\; t\in[0,t_0).
\end{eqnarray}
After some computations,
\begin{eqnarray}\label{r5.11}
&&\|\varphi\|_2^4 + \|\varphi\|_2^2\langle F, \varphi\rangle_{H^{-1}}  \nonumber\\
&=& \big(\|\varphi\|_2^2+\langle F/2, \varphi\rangle_{H^{-1}}\big)^2-|\langle F/2, \varphi\rangle_{H^{-1}}|^2 \nonumber \\
&\leq& \big(\|\varphi\|_2^2+\langle F/2, \varphi\rangle_{H^{-1}}\big)^2 \\
&\leq&\|\varphi\|^2_{H^{-1}}\big(\|\varphi\|^2_{H^1}+\|F/2\|_{H^{-1}}^2+\langle F,\varphi\rangle\big), \;\;  t\in[0,T].\nonumber
\end{eqnarray}
This, along with (\ref{r5.10}), indicates
\begin{eqnarray*}
\chi'(t)\le \frac{2}{\|\varphi\|^2_{H^{-1}}}\|{F/2}\|_{H^{-1}}^2,\;\;t\in[0,t_0).
\end{eqnarray*}
Using (\ref{4.2}), and solving this inequality, we obtain
 \begin{eqnarray}\label{r5.12}
\chi(t)\le e^{\frac{L_M^2t}{2}}\chi(0),\;\;t\in[0,t_0).
 \end{eqnarray}

Next, by the second equation of (\ref{r5.9}) and (\ref{4.2}),  we have
   \begin{eqnarray*}
 &&0= \frac{1}{2}\frac{d}{dt}\|\varphi\|_{H^{-1}}^2 +\chi(t) \|\varphi\|_{H^{-1}}^2 + |\langle F-\varphi, \varphi\rangle_{H^{-1}}|\\
  &&\quad\le \frac{1}{2}\frac{d}{dt}\|\varphi\|_{H^{-1}}^2 +\chi(t) \|\varphi\|_{H^{-1}}^2 +  \|\varphi\|_{H^{-1}}(\|F\|_{H^{-1}}-\|\varphi\|_{H^{-1}})\\
   &&\quad\le \frac{1}{2}\frac{d}{dt}\|\varphi\|_{H^{-1}}^2 +e^{\frac{L_M^2t}{2}}\chi(0)\|\varphi\|_{H^{-1}}^2 +  L_M\|\varphi\|_{2}\|\varphi\|_{H^{-1}}, \;\;t\in[0,t_0),
 \end{eqnarray*}
 which, together with (\ref{r5.12}), yields
  \begin{eqnarray*}
 &&0\leq{1\over 2}{d\over dt}\|\varphi\|_{H^{-1}}^2 +e^{\frac{L_M^2T}{2}} \big(\chi(0) + L_M\sqrt{\chi(0)}\big)\|\varphi\|_{H^{-1}}^2,\;\;t\in[0,t_0).
 \end{eqnarray*}
Solving this inequality, we obtain
\begin{eqnarray}\label{r5.13}
 \|\varphi(\cdot,0)\|_{H^{-1}}^2\le e^{2e^{\frac{L_M^2T}{2}}\big(\chi(0) + L_M\sqrt{\chi(0)}\big)T}\|\varphi(\cdot,t)\|_{H^{-1}}^2,\;\; t\in[0,t_0).
\end{eqnarray}

Then, it follows  from (\ref{r5.13}) that when $t\in[0,t_0)$,
 \begin{eqnarray*}
\frac{\|\varphi(\cdot, 0)\|_{2}^2}{\|\varphi(\cdot, t)\|_{2}^2}\leq \frac{\|\varphi(\cdot, 0)\|_{H^{-1}}^2}{\|\varphi(\cdot, t)\|_{H^{-1}}^2}\chi(0)
\leq e^{2e^{\frac{L_M^2T}{2}}\big(\chi(0) + L_M\sqrt{\chi(0)}\big)T}\chi(0).
 \end{eqnarray*}
Obviously,  we have
 \begin{eqnarray*}
\sqrt{\chi(0)}\leq\chi(0) \textrm{ and }  \chi(0)\leq e^{\chi(0)}.
 \end{eqnarray*}
Therefore,
  \begin{eqnarray}\label{r5.14}
 \frac{\|\varphi(\cdot, 0)\|_{2}^2}{\|\varphi(\cdot, t)\|_{2}^2}\leq  e^{2e^{\frac{L_M^2T}{2}}\big(\chi(0) + L_M\sqrt{\chi(0)}\big)(1+T)},\:\:t\in[0,t_0).
 \end{eqnarray}
This leads to a contradiction with (\ref{r5.7}).
 Hence, we prove that $\varphi(\cdot,t)\neq0$ for each $t\in(0,T]$.\\

Second, we are going to prove (\ref{1.4}).
We have the function $t\rightarrow{\|\varphi(\cdot, t)\|_2^2\over \|\varphi(\cdot, t)\|^2_{H^{-1}}}$ is well defined over $[0,T]$.
Applying the same argument, we can verify that
   \begin{eqnarray*}
\|\varphi(\cdot, 0)\|_{2}^2\leq  e^{2e^{\frac{L_M^2T}{2}}\big(\chi(0) + L_M\sqrt{\chi(0)}\big)(1+T)}\|\varphi(\cdot, t)\|_{2}^2,\:\:t\in[0,T],
 \end{eqnarray*}
which, together with (\ref{1.3}), gives (\ref{1.4}).
 Hence, we  complete the proof.
\end{proof}

\subsection{Proof of Theorem \ref{theorem2}}
\begin{proof}
Let $y_i$ $(i=1,2)$ be the solution to (\ref{1.1}) with initial value $y_i^0$ and let $\varphi=y_1-y_2$.
Take $\delta\in(0, T)$ and write $T_\delta=T-\delta$.  We first define a function  $\varphi_{\delta}:=\varphi(x,t+\delta)$,
  where  $(x,t)\in\Omega\times[0,T_\delta]$.

Applying Theorem \ref{theorem3.2}, $y_i^0\in L^2(\R^n)$  ($i=1,2$, and $1<p<\frac{4}{n}+1$),
we have  $y_i\in C([0,T]; L^2(\mathbb{R}^n))\cap L_{loc}^\infty((0,T); L^\infty(\mathbb{R}^n))$. Thus,
$y_i(\cdot, \delta)\in L^2(\R^n)\cap L^\infty(\R^n)$ $(i=1,2).$
It follows from  Theorem \ref{theorem1} that\\
 $(i)$ There are constants  $\beta=\beta(L, r)\in (0,1)$ and $C=C(L, r, T_{\delta}, L_{M_\delta})>0$
  such that
  \begin{eqnarray}\label{r5.15}
\int_{\R^n}|\varphi_\delta(x,T_\delta)|^2dx
&\leq& \!\!C
\big(\int_{\R^n}|\varphi_{\delta}(\cdot,0)|^2dx\big)^{1-\beta}\big(\int_{\R^n}|\varphi_\delta(x,T_\delta)|^2dx\big)^\beta,
\end{eqnarray}
$(ii)$ If $\varphi_\delta(\cdot,0)\neq0$, then there exist a positive number  $C=C(L, r, T, L_{M_\delta})$ such that
\begin{eqnarray}\label{r5.16}
\int_{\R^n}|\varphi_{\delta}(\cdot,0)|^2dx &\leq& Ce^{C
\frac{\|\varphi_{\delta}(\cdot,0)\|_{2}^2}{\|\varphi_{\delta}(\cdot,0)\|_{H^{-1}}^2}}\int_{\omega}|\varphi_{\delta}(x,T_\delta)|^2dx.
\end{eqnarray}
where $M_\delta=\max\{\|y_i\|_{L^\infty(\delta, T; L^{\infty}(\R^n))}\:|\:i=1,2\}$ and $L_{M_\delta}>0$
is the Lipschitz constant of $f$ on the domain
$\mathcal{D}_{M_\delta}=\{s\in \mathbb{R}\:|\: |s|\leq M_{\delta} \}$.
Taking $\delta=\frac{T}{2}$ in (\ref{r5.15}), then using (\ref{3.5}), we obtain (\ref{1.5}).

If $\varphi(\cdot,T)=0$ over $\omega$, i.e.,  $\varphi_{\delta}(\cdot,T_\delta)=0$ over $\omega$, then by (\ref{r5.16}),
we have $\varphi_{\delta}(\cdot,0)=0$  over $\R^n$, i.e.,
$\varphi(\cdot,\delta)=0$  over $\R^n$. By Theorem \ref{theorem3.2}, we have $y_i\in C([0,T]; L^2(\R^n))$. Letting $\delta\rightarrow0+$, we obtain $\varphi(\cdot,0)=0$, i.e. $y_1^0=y_2^0$.
 We complete the proof.
\end{proof}

\section{Appendix A}
\subsection{Proof of Theorem \ref{theorem3.1}}

\begin{proof}
Let  $T^*>0$, which will be chosen later and let $\varrho=2\cdot\min\{\|y_0\|_{q}, \|y_0\|_{\infty} \}$. Assume $y_0\in L^q(\mathbb{R}^n)\cap L^\infty(\mathbb{R}^n)$.
Define the function spaces:
 $\mathcal{A}:=C([0,T^*]; L^q(\mathbb{R}^n))\cap L^\infty((0,T^*); L^\infty(\mathbb{R}^n))$ and
  $\mathcal{B}:=\{\xi\in\mathcal{A}\: |\: \|\xi(t)\|_{q}\leq \varrho \textrm{ and } \|\xi(t)\|_{\infty}\leq \varrho \textrm{ for } t\in[0,T*]\}.$
 The proof will be split into two steps.

 \noindent{\it Step 1. We first prove
that for some $T^*>0$, equation (\ref{1.1}) admits a solution  $y\in\mathcal{A}$.}

 For any $\xi\in \mathcal{B}$, we define $y(t;\xi):=e^{t\triangle}y_0+\int_0^te^{(t-s)\triangle}f(\xi(s))ds,  \: t\in[0,T^*].$
 Clearly, $y(\cdot;\xi)\in \mathcal{A}$. We equip $\mathcal{B}$ with the distance
$d(\xi,\eta)=\sup_{0<t<T^*}\|\xi(t)-\eta(t)\|_{\infty}.$

 Next, we define a map $\Lambda:\:\mathcal{B}\mapsto \mathcal{A}$ by
$\Lambda(\xi):=y(\cdot;\xi),\;\mbox{ for }\; \xi\in \mathcal{B}.$
It follows from $(A_2)$ that
$|f(y)|=|f(y)-f(0)|\leq L_{\varrho}|y|,$
where $L_{\varrho}$ is the Lipschitz constant of $f$ on the set $\{\tau\in\mathbb{R}\: |\: |\tau|\leq\varrho\}$.
By ($A_2$) and Lemma \ref{lemma2.1}, we have
\begin{eqnarray}\label{3.1}
\|y(t;\xi)\|_q
&\leq& \|e^{t\triangle}y_0\|_q+\int_0^t\|e^{(t-s)\triangle}f(\xi(s))\|_{q}ds \nonumber\\
&\leq&\|y_0\|_q+L_{\varrho}T^*\sup_{0<t<T^*}\|\xi(t)\|_{q}\nonumber\\
&\leq&\|y_0\|_q+L_{\varrho}T^*\varrho,\:\:t\in[0,T^*].
\end{eqnarray}
Then,
 \begin{eqnarray}\label{3.2}
\|y(t;\xi)\|_{\infty}
&\leq& \|e^{t\triangle}y_0\|_{\infty}+\int_0^t\|e^{(t-s)\triangle}f(\xi(s))\|_{\infty}ds \nonumber\\
&\leq&\|y_0\|_\infty+\int_0^t \|f(\xi(s))\|_{\infty}ds \nonumber\\
&\leq&\|y_0\|_\infty+L_{\varrho}T^*\varrho,\:\:t\in[0,T^*].
\end{eqnarray}
Finally, we obtain for $\xi,\eta\in\mathcal{B}$ and $t\in[0,T^*]$
 \begin{eqnarray}\label{3.3}
\|y(t;\xi)-y(t;\eta)\|_{\infty}
&\leq&\int_0^t\|e^{(t-s)\triangle}(f(\xi(s))-f(\eta(s)))\|_{\infty}ds \nonumber\\
&\leq&\int_0^t \|f(\xi(s))-f(\eta(s))\|_{\infty}ds \nonumber\\
&\leq&L_\varrho T^*d(\xi,\eta), \:\:t\in[0,T^*].
\end{eqnarray}
Applying (\ref{3.1}), (\ref{3.2}), and (\ref{3.3}), it shows that if $T^*$ is small enough (which only depends on $\|y_0\|_q$, and $\|y_0\|_\infty$),
then $\Lambda(\mathcal{B})\subseteq \mathcal{B}$  and  $\Lambda:\:\mathcal{B}\mapsto \mathcal{B}$ is a strict   contraction.
By the Banach fixed-point theorem, there exists a unique $y\in\mathcal{B}$ satisfying (\ref{1.1}).

\noindent{\it Step 2. We
show that the solution $y$ is unique in $\mathcal{A}$.}

We suppose that $y, \bar{y}\in$ $\mathcal{A}$ are both solutions to equation (\ref{1.1}) with the initial value $y_0$. Write
  \begin{eqnarray*}
\varrho_1:=\max \{\|y\|_{L^\infty((0,T^*); L^\infty(\mathbb{R}^n))}, \|\bar{y}\|_{L^\infty((0,T^*); L^\infty(\mathbb{R}^n))}\}.
 \end{eqnarray*}
Using $(A_2)$ and Lemma \ref{lemma2.1} again, we have that for $t\in(0, T^*]$,
\begin{eqnarray*}
\|y(t)-\bar{y}(t)\|_\infty&=&\|\int_{0}^t e^{(t-s)\triangle}[f\big(y(t)\big)-f\big(\tilde{y}(t)\big)]ds\|_\infty\\
&\leq& L_{\varrho_1}\int_{0}^t\|y-\bar{y}\|_{\infty}ds
 \end{eqnarray*}
where $L_{\varrho_1}$ is the Lipschitz constant of $f$ on the set $\{\tau\: |\: |\tau|\leq \varrho_1\}$.
Uniqueness now is obtained by Gronwall's inequality.
 This complete the proof.
\end{proof}

\subsection{Proof of Lemma \ref{lemma2.3}}
\ \ \ \
This part is devoted to the proof of Lemma \ref{lemma2.3}.
We borrow the idea from Theorem A1 in \cite{Brezis}, and first introduce a differential inequality.
\begin{lemma}\label{lemma2.2}
Let $T>0$.
 Suppose that $g(\cdot)\in C^1[0,T]$ with $g(t)\geq0$ for all $t\in[0,T]$. If there exist three positive numbers $\alpha, A, B$ such that
 \begin{eqnarray*}
g'(t)+Ag(t)^{1+\alpha}\leq Bg(t)\;\;\mbox{for each}\;\;t\in[0,T],
 \end{eqnarray*}
 then
  \begin{eqnarray*}
g(t)\leq\big(\frac{1}{\alpha At}\big)^{\frac{1}{\alpha}}e^{Bt}\;\;\mbox{for each}\;\;t\in (0,T].
 \end{eqnarray*}
\end{lemma}
This inequality can be found in \cite{Brezis}.\

\textbf{Proof of Lemma \ref{lemma2.3}.}
\begin{proof}
The proof will be split into the following two steps.

\textit{Step 1:  Given $u_0\in L^\infty(\mathbb{R}^n)$, we will prove that equation (\ref{2.2}) admits a unique
solution $u\in L^\infty((0,T; L^\infty(\mathbb{R}^n))$. Moreover, the solution $u$ satisfies that
\begin{eqnarray}\label{6.5}
\|u(t)\|_{\infty}\leq Ce^{C\mathcal{L}^{\vartheta}t}\|u_0\|_{\infty},\: for\:\: t\in[0,T],
 \end{eqnarray}
 where the numbers $\mathcal{L}, \vartheta$ are  given in equation (\ref{2.3}), and the positive number $C$ only depends on $(n,\sigma, \gamma)$. }

We still write $\mathcal{A}:=L^\infty(0, T; L^\infty(\mathbb{R}^n))$. For each $\xi\in \mathcal{A}$, we define function $u(\cdot;\xi)$ as follows:
\begin{eqnarray*}
u(t;\xi):=S(t)u_0+\int_0^tS(t-s)(a\xi) ds,\:\:t\in[0,T].
\end{eqnarray*}
By Lemma \ref{lemma2.1}, we observe that $u(\cdot;\xi)\in\mathcal{A}$.
Define a map $\Psi: \mathcal{A}\rightarrow\mathcal{A}$ by setting $\Psi(\xi):=u(\cdot;\xi)$.
For $\xi_i\in\mathcal{A}$ $(i=1,2)$ and $t\in[0,T]$, we have
\begin{eqnarray*}
\|\Psi_1(\xi_2)(t)-\Psi_1(\xi_1)(t)\|_{\infty}
=\|\int_0^tS(t-s)\big(a(\xi_2-\xi_1)\big)ds\|_{\infty}.
\end{eqnarray*}
 Using Lemma \ref{lemma2.1} and $\sigma>\frac{n}{2}$, we obtain
\begin{eqnarray}\label{6.3}
\|\Psi_1(\xi_2)-\Psi_1(\xi_1)\|_{\mathcal{A}}\leq\mathcal{L}\frac{T^{1-\kappa}}{1-\kappa}\|\xi_2-\xi_1\|_{\mathcal{A}},
\end{eqnarray}
where  $\kappa=\frac{n}{2\sigma}$.
We first consider the case that
\begin{eqnarray}\label{6.4}
\mathcal{L}\mathcal{}\frac{T^{1-\kappa}}{1-\kappa}<\frac{1}{2}.
\end{eqnarray}
Applying (\ref{6.3}) and (\ref{6.4}), we deduce $\Psi$ is a strict contraction map on $\mathcal{A}$. Thus, there exists a unique solution $u\in\mathcal{A}$
to equation (\ref{2.2}).
It is easy to find $\Psi(0)=S(t)u_0\in \mathcal{A}$. Then
 by  taking $\xi_1=u$ and $\xi_2=0$ in (\ref{6.3}), we have
  \begin{eqnarray*}
\|u\|_{\mathcal{A}}=\|\Psi(u)\|_{\mathcal{A}}\leq \|\Psi(u)-\Psi(0)\|_{\mathcal{A}}+\|\Psi(0)\|_{\mathcal{A}}.
 \end{eqnarray*}
which, together with Lemma \ref{lemma2.1} and (\ref{6.3}), leads to
  \begin{eqnarray*}
\|u\|_{\mathcal{A}}\leq 2\|u_0\|_{\infty}.
 \end{eqnarray*}
Thus, (\ref{6.5}) holds in this case.

Secondly, if (\ref{6.4}) does not hold, we choose another $T_0\in[0,T]$ such that $\mathcal{L}\frac{T_0^{1-\kappa}}{1-\kappa}<\frac{1}{2}$,
and work on this problem in $[0,T_0]$. Then, by a standard iteration argument,  we can get the desired results in the second case.

\textit{Step 2: We will prove (\ref{2.3a}).}

By the duality technique (see Section 2 of \cite{Daners}), we can obtain that there exists $C=C(n,\sigma, \gamma)$ such that
\begin{eqnarray}\label{6.7}
\|u(t)\|_{1}\leq Ce^{C\mathcal{L}^{\vartheta}t}\|u_0\|_{1},
 \end{eqnarray}
 where constants $\mathcal{L}$, $\vartheta$ are given in (\ref{2.3}).
 This, along with (\ref{6.5}),  and  Riesz-Thorin's interpolation theorem, shows
 \begin{eqnarray}
\|u(t)\|_{p}\leq Ce^{C\mathcal{K}^{\vartheta}t}\|u_0\|_{p}, \:\textrm{for each}\:1\leq p\leq\infty.
 \end{eqnarray}

   We assume that $u_0\neq0$ (In fact, Lemma \ref{lemma2.2} is valid when $u_0=0$). Multiplying equation (\ref{2.2})  by $u$  and integrating it over $\mathbb{R}^n$,  we obtain
 \begin{eqnarray}\label{6.6}
&&\frac{1}{2}\frac{d}{dt}\int_{\mathbb{R}^n}|u|^2dx+\int_{\mathbb{R}^n}|\nabla u|^2dx=-\int_{\mathbb{R}^n}au^2dx.
 \end{eqnarray}
Let $\sigma'\in[1,+\infty]$ satisfy $\frac{1}{\sigma}+\frac{1}{\sigma'}=1$. Combining with H\"{o}lder and Gagliardo-Nirenberg's inequalities, yields
 \begin{eqnarray*}
\int_{\mathbb{R}^n}|a||u|^2dx
\leq\|a\|_{\sigma}\|u\|_{2\sigma'}^2\leq C(n)\|a\|_{\sigma}(\int_{\mathbb{R}^n}|\nabla u|^2dx)^{\frac{n}{2\sigma}}
(\int_{\mathbb{R}^n}|u|^2dx)^{\frac{2\sigma-n}{2\sigma}}.
 \end{eqnarray*}
 Then, by Young's inequality, it yields
 \begin{eqnarray*}
 C(n)\|a\|_{\sigma}(\int_{\mathbb{R}^n}|\nabla u|^2dx)^{\frac{n}{2\sigma}}
(\int_{\mathbb{R}^n}|u|^2dx)^{\frac{2\sigma-n}{2\sigma}}&\leq& C(n)(\int_{\mathbb{R}^n}|\nabla u|^2dx)^{\frac{n}{2\sigma}} (\mathcal{L}^{\vartheta}\int_{\mathbb{R}^n}|u|^2dx)^{\frac{2\sigma-n}{2\sigma}}\\
&\leq&\frac{1}{2}\int_{\mathbb{R}^n}|\nabla u|^2dx+C(n)\mathcal{L}^{\vartheta}\int_{\mathbb{R}^n}|u|^2dx.
 \end{eqnarray*}
Therefore,
 \begin{eqnarray*}
|\int_{\mathbb{R}^n} au^2dx|
\leq\frac{1}{2}\int_{\mathbb{R}^n}|\nabla u|^2dx+C(n)\mathcal{L}^{\vartheta}\int_{\mathbb{R}^n}|u|^2dx.
 \end{eqnarray*}
It, along with (\ref{6.6}), indicates
 \begin{eqnarray*}
\frac{d}{dt}\int_{\mathbb{R}^n}|u|^2dx+\int_{\mathbb{R}^n}|\nabla u|^2dx
\leq C(n)\mathcal{L}^{\vartheta}\int_{\mathbb{R}^n}|u|^2dx.
 \end{eqnarray*}
Using Gagliardo-Nirenberg's inequality and (\ref{6.7}), it indicates that
 \begin{eqnarray*}
\big(\int_{\mathbb{R}^n}|u|^2dx\big)^{\frac{n+2}{n}}\leq(Ce^{C\mathcal{L}^{\vartheta}t}\|u_0\|_{1})^{\frac{4}{n}}\int_{\mathbb{R}^n}|\nabla u|^2dx,
 \end{eqnarray*}
 where $C=C(n,\sigma, \gamma)$.
Let $g(t):=\int_{\mathbb{R}^n}|u|^2(x,t)dx$, $t\in[0,T]$. Clearly, $g(t)$ satisfies the following differential inequality:
 \begin{eqnarray*}
g'(t)+Ag(t)^{1+\frac{2}{n}}\leq Bg(t),
 \end{eqnarray*}
where $A=(Ce^{C\mathcal{L}^{\vartheta}t}\|u_0\|_{1})^{-\frac{4}{n}}$ and $B=C\mathcal{L}^{\vartheta}$.
By  Lemma \ref{lemma2.2}, we obtain
  \begin{eqnarray*}
g(t)\leq\big(\frac{n}{2At}\big)^{\frac{n}{2}}e^{Bt}.
 \end{eqnarray*}
Thus,
\begin{eqnarray}\label{6.9}
 \|u\|_{2}\leq Ce^{C\mathcal{L}^{\vartheta}t}t^{-\frac{n}{4}} \|u_0\|_{1}.
 \end{eqnarray}
Using the duality technique again, we obtain
\begin{eqnarray*}
 \|u\|_{\infty}\leq Ce^{C\mathcal{L}^{\vartheta}t}t^{-\frac{n}{4}} \|u_0\|_{2}.
 \end{eqnarray*}
This, along with (\ref{6.9}), shows
\begin{eqnarray*}
 \|u\|_{\infty}\leq Ce^{C\mathcal{L}^{\vartheta}t}t^{-\frac{n}{2}}\|u_0\|_{1},
 \end{eqnarray*}
 where the number $C$ only depends on $(n,\sigma, \gamma)$.
It, together with (\ref{6.5}) and Riesz-Thorin's interpolation theorem, indicates (\ref{2.3a}).
This completes the proof.
\end{proof}

\subsection{Proof of Theorem \ref{theorem3.2}}
At present, we have not found a complete proof of  theorem \ref{theorem3.2}.
 The first part follows from  \cite{Miao} (see Section 2.2 in \cite{Miao}). For completeness, we present the details below.
\begin{proof}
Let  $T^*>0$, which will be chosen later.  Define
$\mathcal{E}:=L^\infty((0,T^*);L^q(\mathbb{R}^n))\cap L_{loc}^\infty((0,T^*);L^{pq}(\mathbb{R}^n)),$
and
$\mathcal{B}:=\{\xi\in \mathcal{E} \:  |  \:\|\xi(t)\|_q\leq 2\|y_0\|_{q} \:\textrm{ and }\:t^{\alpha}\|\xi(t)\|_q\leq 2\|y_0\|_{q},\:\forall t\in[0,T^*]\},$
where $\alpha=\frac{n(p-1)}{2pq}.$ Obviously, $0<\alpha<\frac{1}{p}<1.$
We equip $\mathcal{B}$ with the distance
$d(\xi,\eta)=\sup_{0<t<T^*}t^\alpha\|\xi(t)-\eta(t)\|_{pq}.$
The following proof will be split into three steps.

\noindent{\it Step 1. We prove
that for some $T^*>0$, equation (\ref{1.1}) admits  a solution  $y\in C([0,T^*]; L^q(\mathbb{R}^n)).$}

For each $\xi\in \mathcal{B}$, we define functions $y(\cdot;\xi)$ as follows:
\begin{eqnarray*}
y(t;\xi):=e^{t\triangle}y_0+\int_0^te^{(t-s)\triangle}f(\xi(s))ds,\:t\in[0,T^*].
\end{eqnarray*}
It is obvious that $y(\cdot;\xi)\in \mathcal{E}$. Now, we define a map $\Phi:\:\mathcal{B}\mapsto \mathcal{E}$ by
$\Phi(\xi):=y(\cdot;\xi),\;\;\mbox{for}\;\; \xi\in \mathcal{B}.$
Applying ($A_3$) and Lemma \ref{lemma2.1}, we first have
\begin{eqnarray}\label{2.4}
\|y(t;\xi)\|_q
&\leq& \|e^{t\triangle}y_0\|_q+\int_0^t\|e^{(t-s)\triangle}f(\xi(s))\|_{q}ds \nonumber\\
&\leq&\|y_0\|_q+\int_0^t \|f(\xi(s))\|_{q}ds \nonumber\\
&\leq&\|y_0\|_q+C\big(\sup_{0<t<T^*}t^{\alpha}\|\xi(t)\|_{pq}\big)^p\int_0^t s^{-\alpha p}ds\nonumber\\
&\leq&\|y_0\|_q+C\frac{T^{*1-\alpha p}}{1-\alpha p}\big(2\|y_0\|_q\big)^p,\:\:t\in[0,T^*].
\end{eqnarray}
Then,
 \begin{eqnarray}\label{2.5}
t^{\alpha}\|y(t;\xi)\|_{pq}
&\leq& t^{\alpha}\|e^{t\triangle}y_0\|_{pq}+t^{\alpha}\int_0^t\|e^{(t-s)\triangle}f(\xi(s))\|_{pq}ds \nonumber\\
&\leq&\|y_0\|_q+t^{\alpha}\int_0^t (t-s)^{-\alpha}\|f(\xi(s))\|_{q}ds \nonumber\\
&\leq&\|y_0\|_q+Ct^{\alpha}\big(\sup_{0<t<T^*}t^{\alpha}\|\xi(t)\|_{pq}\big)^p\int_0^t (t-s)^{-\alpha}s^{-\alpha p}ds\nonumber\\
&\leq&\|y_0\|_q+CT^{*1-\alpha p}\big(2\|y_0\|_q\big)^p\cdot\int_0^1 (1-\rho)^{-\alpha}\rho^{-\alpha p}d\rho,\:\:t\in[0,T^*].
\end{eqnarray}
By the same argument, we obtain for $\xi,\eta\in\mathcal{B}$ and $t\in[0,T^*]$
 \begin{eqnarray}\label{2.5aa}
t^{\alpha}\|y(t;\xi)-y(t;\eta)\|_{pq}
&\leq&t^{\alpha}\int_0^t\|e^{(t-s)\triangle}(f(\xi(s))-f(\eta(s)))\|_{pq}ds \nonumber\\
&\leq&t^{\alpha}\int_0^t (t-s)^{-\alpha}\|f(\xi(s))-f(\eta(s))\|_{q}ds \nonumber\\
&\leq&Ct^{\alpha}\sup_{0<t<T^*}\big(\|\xi(t)\|_{pq}^{p-1}+\|\eta(t)\|_{pq}^{p-1}\big)\int_0^t (t-s)^{-\alpha}s^{-\alpha}ds\cdot d(\xi,\eta)\nonumber\\
&\leq&CT^{*1-\alpha p}\big(2\|y_0\|_q\big)^{p-1}\cdot\int_0^1 (1-\rho)^{-\alpha}\rho^{-\alpha}d\rho\cdot d(\xi,\eta).
\end{eqnarray}
Using (\ref{2.4}), (\ref{2.5}), and (\ref{2.5aa}), it shows that if $T^*$ is small enough (which only depends on $\|y_0\|_q, p$, and $q$)
then $\Phi(\mathcal{B})\subseteq \mathcal{B}$  and  $\Phi:\:\mathcal{B}\mapsto \mathcal{B}$ is a strict   contraction.
By the Banach fixed-point theorem, there exists a unique $y\in\mathcal{B}$ satisfying (\ref{1.1}).

Now, we are going to prove $y\in C([0,T^*];L^q(\mathbb{R}^n)).$ Since $y\in\mathcal{B}$, $\alpha p<1$, and ($A_3$) holds,
we deduce $f(y)\in L^1((0,T^*); L^q(\mathbb{R}^n))$.
It follows from $y(t;\xi):=e^{t\triangle}y_0+\int_0^te^{(t-s)\triangle}f(y(s))ds$ that $y\in C([0,T^*];L^q(\mathbb{R}^n))$
 and it is a solution of (\ref{1.1}).


\noindent{\it Step 2. We prove (\ref{3.5}), which shows $y\in L_{loc}^\infty((0,T^*); L^\infty(\mathbb{R}^n))$.}\\
 There are only two possibilities:\\
\textbf{Case 1: $q>\frac{n(p-1)}{2}$, $q\geq1$, and $q\geq p-1$.\\
Case 2: $q>\frac{n(p-1)}{2}$, and $1\leq q<p-1$.\\}
Here we will only give the prove of Case 2, and  another case can be easily obtained  using Lemma \ref{lemma2.3} (refer to Theorem 3.2 in \cite{G. Zheng}).
Similar to the proof of (\ref{2.4}) and (\ref{2.5}), we have
\begin{eqnarray}\label{2.15aa}
\|y(t)\|_q+t^{\alpha}\|y(t)\|_{pq}\leq C\|y_0\|_q,\:\:t\in[0,T^*],
\end{eqnarray}
where $\alpha=\frac{n(p-1)}{2pq}$, and $y(\cdot)$ is the solution of (\ref{1.1}) obtained in the first step with initial data $y_0$.
Define a function as follows:
\begin{eqnarray*}
 a(x,t):=
  \left\{
\begin{array}{ll}
\frac{f(y)}{y},&  u(x,t)\neq0,\\
f'(0),&  \textrm{otherwise},
\end{array}
\right.
 \end{eqnarray*}
where $t\in[0,T^*]$.
By direct computations,  we have $f(y)=a(x,t)y$ and $a(x,t)\in L^\infty((\frac{t}{2},t);L^{\sigma}(\mathbb{R}^n))$ with  $\sigma=\frac{pq}{p-1}$
 ($\sigma\geq1$ and $\sigma>\frac{n}{2}$). Then, it follows from (\ref{2.15aa}) that
\begin{eqnarray*}
\|a\|_{L^\infty((\frac{t}{2},t); L^{\sigma})}\leq C\sup_{\frac{t}{2}<s<t}\|y(s)\|_{pq}^{p-1}\leq
C\sup_{\frac{t}{2}<s<t}s^{-\alpha(p-1)}\|y_0\|_q^{p-1}\leq Ct^{-\alpha(p-1)}\|y_0\|_q^{p-1}.
\end{eqnarray*}
Then, we apply Lemma \ref{lemma2.3} on the interval $(\frac{t}{2},t)$ with $\gamma=pq$, and obtain
\begin{eqnarray}
\|y(t)\|_{\infty}&\leq& C\exp( Ct(t^{-\alpha(p-1)\vartheta}\|y_0\|_q^{(p-1)\vartheta}) )t^{-\frac{n}{2\gamma}}\|y(\frac{t}{2})\|_{pq} \nonumber \\
&\leq& C\exp( Ct(t^{-\alpha(p-1)\vartheta}\|y_0\|_q^{(p-1)\vartheta}))t^{-\frac{n}{2\gamma}}t^{-\alpha}\|y_0\|_{q}\nonumber \\
&\leq& C\exp( Ct(t^{-\alpha(p-1)\vartheta}\|y_0\|_q^{(p-1)\vartheta}))t^{-\frac{n}{2q}}\|y_0\|_{q}
\end{eqnarray}
It, along with (\ref{2.15aa}), leads to (\ref{3.5}) and
 $y\in L_{loc}^\infty((0,T^*); L^\infty(\mathbb{R}^n))$.


\noindent{\it Step 3. We
show that the solution of  (\ref{1.1}) is unique in $C([0,T^*]; L^q(\mathbb{R}^n))\cap L_{loc}^\infty((0,T^*); L^\infty(\mathbb{R}^n))$.}

Write $\mathcal{A}:=C([0,T^*]; L^q(\mathbb{R}^n))\cap L_{loc}^\infty((0,T^*); L^\infty(\mathbb{R}^n))$. We first suppose that $\bar{y}(\cdot)\in\mathcal{A}$  is another  solution to (\ref{1.1}) and let $U(\cdot)y_0$ is the solution constructed
 by the step 1 with the initial data $y_0$.
  There are only two possibilities:\\
\textbf{Case 1: $y_0\in L^q(\mathbb{R}^n)\cap L^\infty(\mathbb{R}^n)$.} Using  Theorem \ref{theorem3.1}, we obtain  $\bar{y}(\cdot)=U(\cdot)y_0$. Namely, $\bar{y}(\cdot)$ is the solution constructed via the above contraction argument.\\
\textbf{Case 2: $y_0\in L^q(\mathbb{R}^n)$, and $y_0\notin L^\infty(\mathbb{R}^n)$.}  Let
 \begin{eqnarray*}
K=\{\bar{y}(t)\:|\:t\in[0,T^*]\}, \: \textrm{ and }\: M=\sup_{t\in[0,T^*]}\|\bar{y}(t)\|_{q}.
\end{eqnarray*}
Similar to the calculation in (\ref{2.4}), (\ref{2.5}), and (\ref{2.5aa}), we obtain there exists a uniform $T_1>0$ such
that $U(t)\bar{y}(s)$ (where $s\in [0,T^*]$) is well define for all $t\in[0,T_1]$. Moreover, we have
\begin{eqnarray}\label{5.18}
\|U(t)\bar{y}(s)\|_q\leq 2M\: \textrm{ and } \: t^{\alpha}\|U(t)\bar{y}(s)\|_{pq}\leq 2M,
\end{eqnarray}
where $\alpha=\frac{n(p-1)}{2pq},$ $s\in [0,T^*]$ and $t\in [0,T_1]$. Since $\bar{y}(s)\in L^\infty(\mathbb{R}^n)$ (where $s\in(0,T^*]$), it follows from
\textbf{Case 1} that $\bar{y}(t+s)=U(t)\bar{y}(s)$ for $0\leq t\leq\min\{T_1, T^*-s\}$.
It, together with (\ref{5.18}) deduce
\begin{eqnarray}\label{5.19}
\|\bar{y}(t+s)\|_q\leq 2M\: \textrm{ and } \: t^{\alpha}\|\bar{y}(t+s)\|_{pq}\leq 2M,
\end{eqnarray}
where $t\in [0,T_1]$ and  $t+s\leq T^*$. Taking $s=t$ in (\ref{5.19}), we deduce
 \textbf{there exists two numbers $\delta>0$  and $T_0>0$ such that $t^{\alpha}\|\bar{y}(t)\|_{pq}\leq\delta$ ( $\alpha=\frac{n(p-1)}{2pq}$) for all $t\in(0,T_0]$.}

Then, by substituting $y$ for $\xi$ and $\bar{y}$ for $\eta$ in (\ref{2.5aa}), and then making use of the standard Gronwall's inequality, we can obtain the uniqueness. This completes the proof.
\end{proof}


\end{document}